\documentclass[11pt, reqno]{amsart}
\usepackage{amsmath, amssymb, latexsym, enumerate,  graphicx, MnSymbol}
\usepackage{amscd,mathrsfs,epic,empheq,float}
\usepackage{bbm}
\usepackage[all]{xy}
 \usepackage{pdfsync}
 \usepackage{todonotes}
\usepackage{tikz}
\usetikzlibrary{arrows}
\usepackage{graphicx}
\usepackage[vcentermath]{youngtab}
\usepackage{multirow}
\usepackage{stmaryrd}
\usepackage{youngtab}
\usepackage{young}
\usepackage{xcolor}

\theoremstyle{definition}
\newtheorem{thm}{Theorem}
\newtheorem{prop}[thm]{Proposition}

\newtheorem{lem}[thm]{Lemma}
\newtheorem{defn}[thm]{Definition}

\newtheorem{rem}[thm]{Remark}
\newtheorem{ex}[thm]{Example}

\newtheorem*{first-step}{First Step}
\newtheorem*{third-step}{Third Step}
\newtheorem*{second-step}{Second Step}

\newtheorem{conj}{Conjecture}

\newtheorem*{main idea}{Main idea}
\newtheorem*{mainresult}{Main result}
\newtheorem*{naitosagakii}{Naito-Sagaki Conjecture}

\newtheorem{cl}{Claim}

\newcommand{\op}{\operatorname}
\newcommand{\tightoverset}[2]{%
  \mathop{#2}\limits^{\vbox to -.5ex{\kern-0.75ex\hbox{$#1$}\vss}}}
  
\newcommand{\tightunderset}[2]{%
  \mathop{#2}\limits_{\vbox to -.5ex{\kern-0.75ex\hbox{$#1$}\vss}}}

\newcommand\smvee{\raise0.9ex\hbox{$\scriptscriptstyle\vee$}}

\newcommand{\E}{{\rm E}}

\newcommand{\s}{{\rm S}}

\newcommand{\N}{{\rm N}}

\newcommand{\e}{\varepsilon}
\newcommand{\p}{{\rm P}}

\newcommand{\li}{{\rm L}}

\newcommand{\<}{\langle}
\newcommand{\rr}{\rangle}

\date{}                                           
\newcount\cols
{\catcode`,=\active\catcode`|=\active
 \gdef\Young#1{\hbox{$\vcenter
 {\mathcode`,="8000\mathcode`|="8000
  \def,{\global\advance\cols by 1 &}%
  \def|{\cr
        \multispan{\the\cols}\hrulefill\cr
        &\global\cols=2 }%
  \offinterlineskip\everycr{}\tabskip=0pt
  \dimen0=\ht\strutbox \advance\dimen0 by \dp\strutbox
  \halign
   {\vrule height \ht\strutbox depth \dp\strutbox##
    &&\hbox to \dimen0{\hss$##$\hss}\vrule\cr
    \noalign{\hrule}&\global\cols=2 #1\crcr
    \multispan{\the\cols}\hrulefill\cr%
   }
 }$}}
\gdef\Skew(#1:#2){\hbox{$\vcenter
{\mathcode`,="8000\mathcode`|="8000
  \dimen0=\ht\strutbox \advance\dimen0 by \dp\strutbox
  \def\boxbeg{\vbox
    \bgroup\hrule\kern-0.4pt\hbox to\dimen0\bgroup\strut\vrule\hss$}%
  \def\boxend{$\hss\egroup\hrule\egroup}%
  \def,{\boxend\boxbeg}%
  \def|##1:{\boxend\vrule\egroup\nointerlineskip\kern-0.4pt
    \moveright##1\dimen0\hbox\bgroup\boxbeg}%
  \def\\##1\\##2:{\boxend\vrule\egroup\nointerlineskip\kern-0.4pt
    \kern ##1\dimen0\moveright##2\dimen0\hbox\bgroup\boxbeg}%
  \moveright#1\dimen0\hbox\bgroup\boxbeg#2\boxend\vrule\egroup
 }$}}
}

\title[Littelmann paths and Littlewood-Richardson Sundaram tableaux]{On a conjecture by Naito-Sagaki: Littelmann paths and Littlewood-Richardson Sundaram tableaux}

\author{Jacinta Torres}

\begin{document}
\maketitle

\begin{abstract}
We prove a special case of a conjecture of Naito-Sagaki about a branching rule for the restriction of irreducible representations of $\mathfrak{sl}(2n,\mathbb{C})$ to $\mathfrak{sp}(2n,\mathbb{C})$. The conjecture is in terms of certain Littelmann paths, with the embedding given by the folding of the type $A_{2n-1}$ Dynkin diagram. We propose and motivate an approach to the conjecture in general, in terms of Littlewood-Richardson Sundaram tableaux. 
\end{abstract}

\section*{Introduction}
Consider the automorphism of $\mathfrak{sl}(2n,\mathbb{C})$ induced by the folding of the Dynkin diagram of type $A_{2n-1}$ along the middle vertex.

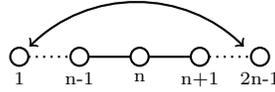
\begin{center}
  \begin{tikzpicture}[scale=.4]
    \foreach \x in {0,...,4}
    \draw[xshift=\x cm,thick] (\x cm,0) circle (.3cm);
    \draw[dotted,thick] (0.3 cm,0) -- +(1.4 cm,0);
    \foreach \y in {1.15,...,2.15}
    \draw[xshift=\y cm,thick] (\y cm,0) -- +(1.4 cm,0);
    \draw[xshift=3.15 cm, dotted,thick] (3.15 cm,0) -- +(1.4 cm,0);
\node (a) at (0,0){};
\node (b) at (7.8 cm,0){};
\node[below] at (0,-0.2) {\tiny{1}};
\node[below] at (2,-0.2) {\tiny{n-1}};
\node[below] at (4,-0.2) {\tiny{n}};
\node[below] at (6,-0.2) {\tiny{n+1}};
\node[below] at (8,-0.2) {\tiny{2n-1}};
\draw[thick, <->] (a) edge[bend left=45] (b);
  \end{tikzpicture}
\end{center}

\noindent The set of $\sigma$-fixed points $\mathfrak{sl}(2n,\mathbb{C})^{\sigma}$ is a sub Lie algebra isomorphic to $\mathfrak{sp}(2n,\mathbb{C})$. Let $\mathfrak{h} \subset \mathfrak{b} \subset \mathfrak{sl}(2n, \mathbb{C})$ be the Cartan subalgebra of diagonal matrices, respectively the Borel subalgebra of upper triangular matrices, and let $\lambda \in \mathfrak{h}^{*}$ be an integral weight that is dominant with respect to this choice. Let $\p_{SSYT}(\lambda)$ be the Littelmann path model for the simple module $\li(\lambda)$ of $\mathfrak{sl}(2n,\mathbb{C})$ associated to the set of semi-standard Young tableaux of shape $\lambda$.  It consists of paths $\pi:[0,1] \rightarrow \mathfrak{h}^{*}_{\mathbb{R}}$, which may be restricted to paths 
$\op{res}(\pi): [0,1] \rightarrow (\mathfrak{h}^{\sigma}_{\mathbb{R}})^{*}$.\\

A distinguishing feature of a Littelmann path model for a simple module $\li(\lambda)$ is the existence of a unique path $\pi^{\lambda}$ contained in the dominant Weyl chamber with endpoint $\pi^{\lambda}(1) = \lambda$. In this spirit, denote the set of restricted paths in $\op{res}(\p_{SSYT}(\lambda))$ that are contained in the corresponding dominant Weyl chamber by $\op{domres}(\lambda)$. 

\begin{naitosagakii}
The following decomposition holds:

\begin{align*}
\op{res}^{\mathfrak{g}}_{\mathfrak{g}^{\sigma}} (\li(\lambda)) = \underset{\delta \in \op{domres}(\lambda)}{\bigoplus} \li(\tightoverset{\hbox{\textasciicircum}}{\delta(1)}).
\end{align*}

\noindent where $\li(\tightoverset{\hbox{\textasciicircum}}{\delta(1)})$ denotes the simple module for $\mathfrak{sl}(2n,\mathbb{C})^{\sigma}$ of highest weight $\tightoverset{\hbox{\textasciicircum}}{\delta(1)}$.
\end{naitosagakii}

\begin{mainresult}
The Naito-Sagaki Conjecture is true for $n=2$, and, for $n \geq 3$, for weights $\lambda$ that belong to the span of the first three fundamental weights. 
\end{mainresult}

It has long been known that branching rules for Levi subalgebras of semi-simple Lie algebras are described by the combinatorics of Littelmann paths. The Lie sub algebra $\mathfrak{sp}(2n,\mathbb{C}) \subset \mathfrak{sl}(2n,\mathbb{C})$ is however not Levi. When the conjecture was originally formulated in \cite{branchconj}, it was proven for weights of hook and rectangular shape, and the proofs used various methods, which were particular to each of the cases. The proof we provide in this paper works in a uniform way for all cases considered in this paper, which were not considered before, and aims to shine light on a new approach towards understanding this conjecture in general. \\

The contents of this paper are organized in seven sections. In Section 1 we introduce basic notation, and in Sections 2 and 3 we recall the combinatorics of tableaux and paths that are needed in order to formulate the Naito-Sagaki conjecture. We state it as  Conjecture \ref{naitosagakiconjecture}, in Section 4. In Section 7 we introduce Littlewood-Richardson Sundaram tableaux, which we will use to show our main result, Theorem \ref{conny2}. In Section 7 we ask some questions related to future work related to Conjecture \ref{naitosagakiconjecture}. 

\subsection*{Acknowledgements}
The author would like to thank St\'ephane Gaussent for introducing her to this problem and Bea Schumann for some very interesting conversations. Thanks also to Daisuke Sagaki for the nice correspondence. The influence of Peter Littelmann and his work are present throughout. Thank you as well to Michael Ehrig, St\'ephane Gaussent and Noah White for their comments on a preliminary version of this paper. The author was financed by the Graduiertenkolleg 1269: Global Structures in Geometry and Analysis, and by the Max Planck Society.

\section{Notation for the Lie algebras}
Let $\mathfrak{h} \subset \mathfrak{sl}(2n, \mathbb{C})$ be the Cartan sub-algebra of diagonal matrices.  Let $\e_{i}$ be the linear map $\mathfrak{h}^{*} \rightarrow \mathbb{C}$ defined by $\op{diag}(a_{1},\hbox{ } \cdots, \hbox{ } a_{2n}) \mapsto a_{i}.$ We write  $\mathfrak{sl}(2n, \mathbb{C}) = \< x_{i}, y_{i}, h_{i} \rr_{ i \in \{1, \cdots, 2n-1\}}$ where $h_{i} = \E_{ii} - \E_{i+1, i+1}$ and where $x_{i}$ and $y_{i}$ are the Chevalley generators corresponding to the simple root $\alpha_{i} := \e_{i} - \e_{i+1}$. The automorphism $\sigma$ is given by

\begin{align*}
\sigma(x_{i}) &= x_{2n-i},\\
 \sigma(y_{i}) &= y_{2n-i}, \hbox{ and }\\
  \sigma(h_{i}) &= h_{2n-i}.
\end{align*}

The fixed point set $\mathfrak{g}^{\sigma}$ is generated as a Lie algebra by $\< {\tightoverset{\hbox{\textasciicircum}}{x}}_{i}, {\tightoverset{\hbox{\textasciicircum}}{y}}_{i}, {\tightoverset{\hbox{\textasciicircum}}{h}}_{i}\rr_{i \in  \{1, \cdots, n\}}$ (see Proposition 7.9 in \cite{kac}), where 

\[
    {\tightoverset{\hbox{\textasciicircum}}{x}}_{i} =  
\begin{cases}
    x_{i} + x_{2n-i}& \text{ if } i \in [1, n) \\
    x_{n}& \text{ if }  i = n
\end{cases}
\]

\[
    {\tightoverset{\hbox{\textasciicircum}}{y}}_{i} =  
\begin{cases}
    y_{i} + y_{2n-i}& \text{ if } i \in [1, n) \\
    y_{n}& \text{ if }  i = n
\end{cases}
\]

\[
    {\tightoverset{\hbox{\textasciicircum}}{h}}_{i} =  
\begin{cases}
    h_{i} + h_{\alpha_{2n-i}}& \text{ if } i \in [1, n) \\
    h_{n}& \text{ if }  i = n.
\end{cases}
\]
\noindent
This Lie algebra is isomorphic to $\mathfrak{sp}(2n, \mathbb{C})$ (see Proposition 7.9 in \cite{kac}) and ${\mathfrak{h}}^{\sigma} = \overset{n}{\underset{ i = 1}{\bigoplus}}{\tightoverset{\hbox{\textasciicircum}}{h}}_{i} = \mathfrak{h} \cap {\mathfrak{g}}^{\sigma}  \subset \mathfrak{h}$ is a Cartan subalgebra. Let $\p_{C_{n}} \subset (\mathfrak{h}^{\sigma})^{*}$ be the set of integral weights of ${\mathfrak{g}}^{\sigma}$ with respect to ${\mathfrak{h}}^{\sigma}$, and let $\p_{A_{2n-1}} \subset \mathfrak{h}^{*}$ be the set of integral weights of $\mathfrak{g}$ with respect to $\mathfrak{h}$.\\

To avoid confusion we will denote elements of $\p_{A_{2n-1}}$ by $\lambda$ and elements of $\p_{C_{n}}$ by $\tightoverset{\hbox{\textasciicircum}}{\lambda}$; in particular, the fundamental weights in  $\p_{A_{2n-1}}$ will be denoted by $\omega_{1}, \cdots, \omega_{2n-1}$, and the fundamental weights in $\p_{C_{n}}$ by ${\tightoverset{\hbox{\textasciicircum}}{\omega}}_{1}, \cdots, {\tightoverset{\hbox{\textasciicircum}}{\omega}}_{n}$. We can then write $$\p_{A_{2n-1}} = \overset{2n-1}{\underset{i = 1}{\bigoplus}}\mathbb{Z}\omega_{i}$$ \noindent and $$\p_{C_{2n}} = \overset{n}{\underset{i = 1}{\bigoplus}}\mathbb{Z}{\tightoverset{\hbox{\textasciicircum}}{\omega}}_{i},$$ \noindent where the direct sums are as $\mathbb{Z}$-modules. Also, for $$\lambda = a_{1}\omega_{1} + \cdots + a_{n}\omega_{n}$$ \noindent we write 
$$\tightoverset{\hbox{\textasciicircum}}{\lambda} = a_{1}{\tightoverset{\hbox{\textasciicircum}}{\omega}}_{1} + \cdots + a_{n}{\tightoverset{\hbox{\textasciicircum}}{\omega}}_{n}$$

\noindent for the corresponding element in $\p_{C_{n}}$. Weights with non-negative coefficients are called \textbf{dominant}. We will denote the corresponding sets by $\p^{+}_{A_{2n-1}}$ and $\p^{+}_{C_{n}}$. We will consider the real vector spaces $\mathfrak{h}_{\mathbb{R}}$ and $\mathfrak{h}^{\sigma}_{\mathbb{R}}$ spanned by the fundamental weights $\omega_{1}, \cdots, \omega_{2n-1}$ respectively ${\tightoverset{\hbox{\textasciicircum}}{\omega}}_{1}, \cdots, {\tightoverset{\hbox{\textasciicircum}}{\omega}}_{n}$. The dominant Weyl chamber is the convex hull of the dominant weights in $\mathfrak{h}^{*}_{\mathbb{R}}$ (respectively $(\mathfrak{h}^{\sigma})^{*}_{\mathbb{R}}$).

\section{Tableaux, words and their paths}
A \textbf{shape} is a finite sequence of non-negative integers $\underline{d} = (d_{1}, \cdots, d_{k})$. An arrangement of boxes of shape $\underline{d}$ is an arrangement of $d_{1}+ \cdots + d_{k}$ columns of boxes such that the first $d_{1}$ columns (from right to left) have one box, and the first $d_{s}$ boxes after the $(d_{1}+ \cdots +d_{s-1})$-th column have $d_{s}$ boxes. To a weight $\lambda = a_{1}\omega_{1}+ \cdots + a_{2n-1}\omega_{2n-1} \in \p_{A_{2n-1}}$ we assign the shape $\underline{d}_{\lambda} = (a_{1}, \cdots, a_{2n-1})$.

\begin{ex} For $\lambda = 3\omega_{1} + \omega_{2}$ we have $\underline{d}_{\lambda} = (3,1)$. To it is associated the following arrangement of boxes:
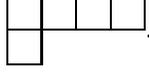
\label{ch2ex1}
\begin{align*}
\Skew(0:,,,|0:).
\end{align*}
\end{ex}

\noindent
A \textbf{semi-standard tableau} of shape $\underline{d}$ is a filling of an arrangement of boxes of shape $\underline{d}$ with letters from the either one of the ordered alphabets 
$$\mathcal{A}_{2n} = \{1< \cdots< 2n\}$$
\noindent or 
$$\mathcal{C}_{n} = \{1< \cdots < n < \bar{n} < \cdots < \bar{1}\}$$ 
\noindent such that entries are strictly increasing downwards along each column of boxes and weakly increasing along each row from left to right. If the entries of a given semi-standard tableau belong to $\mathcal{A}_{2n}$ we will call it a \textbf{semi-standard Young tableau}, and if they belong to  $\mathcal{C}_{n}$, we will call it a \textbf{symplectic semi-standard tableau}. We will denote the set of semi-standard Young tableau of shape $\underline{d}$ by $\Gamma(\underline{d})^{\op{SSYT}}$. 

\begin{ex}
\label{ch2ex2}
A semi-standard Young tableau of shape $(3,1)$:
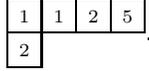
\begin{align*}
\Skew(0:\mbox{\tiny{1}},\mbox{\tiny{1}},\mbox{\tiny{2}},\mbox{\tiny{5}}|0:\mbox{\tiny{2}}).
\end{align*}
\end{ex}


Let $\mathcal{W}(\mathcal{A}_{2n})$ denote the word monoid on $\mathcal{A}_{2n}$ and $\mathcal{W}(\mathcal{C}_{n})$ be the word monoid on $\mathcal{C}_{n}$. The \textbf{word} $w(\mathscr{T})$ of a semi-standard tableau $\mathscr{T}$ is obtained by reading its rows, from right to left and top to bottom. 
 
\begin{ex}
\label{ch2ex3}
 The symplectic semi-standard tableau $\mathscr{T} = \Skew(0:\mbox{\tiny{1}},\mbox{\tiny{1}},\mbox{\tiny{2}},\mbox{\tiny{5}}|0:\mbox{\tiny{$\bar 3$}})$ has word $w(\mathscr{T}) = 5211\bar 3$. 

\end{ex}

\section{Restriction of paths that come from words}
\subsection{Paths and their restrictions}
We will consider paths $\pi:[0,1] \rightarrow \mathfrak{h}_{\mathbb{R}}^{*}$ and 
$\pi':[0,1] \rightarrow (\mathfrak{h}_{\mathbb{R}}^{\sigma})^{*}$ starting at the origin and ending at an integral  weight: 

$$\pi(0) = 0 = \pi'(0), \pi(1) \in \p_{A_{2n-1}}, \pi'(1) \in \p_{C_{n}}.$$

\noindent
The map 

\begin{align*}
\mathfrak{h}^{*} &\rightarrow (\mathfrak{h}^{\sigma})^{*}\\
\varphi & \mapsto {\varphi }|_{\mathfrak{h}^{\sigma}}
\end{align*}
\noindent
induces a map $\op{res}': \p_{A_{2n-1}} \rightarrow  \p_{C_{n}}$. Given a path $\pi: [0,1] \rightarrow \mathfrak{h}_{\mathbb{R}}^{*}$, we define a restricted path $\op{res}(\pi)$ by

\begin{align*}
\op{res}(\pi): [0,1] & \rightarrow (\mathfrak{h}_{\mathbb{R}}^{\sigma})^{*}\\
t & \mapsto \op{res}'(\pi(t)).
\end{align*}

\noindent We will also consider the concatenation $\pi_{2}*\pi_{1}$ of two paths $\pi_{1}$ and $\pi_{2}$ with the same codomain: it is the path obtained by translating $\pi_{2}$ to the endpoint $\pi_{1}(1)$ of $\pi_{1}$. A path is \textbf{dominant} if its image is contained in the dominant Weyl chamber.

\subsection{Paths that come from words}
In the first part of this section we follow Section 2 of \cite{branchconj}. Let $w = w_{1}\hbox{ }\cdots \hbox{ } w_{k}$ be a word, either in $\mathcal{W}(\mathcal{A}_{2n})$ or in $\mathcal{W}(\mathcal{C}_{2n-1})$. To it we assign the path:

\begin{align*}
\pi_{w} = \pi_{w_{k}} * \cdots * \pi_{w_{1}},
\end{align*}

\noindent
where, for $w_{i} \in \mathcal{A}_{2n}$ (respectively $w_{i} \in \mathcal{C}_{n}$), the path $\pi_{w_{i}}: [0,1] \rightarrow \mathfrak{h}_{\mathbb{R}}^{*}$ (respectively $\pi_{w_{i}}: [0,1] \rightarrow (\mathfrak{h}_{\mathbb{R}}^{\sigma})^{*})$ is given by $t \mapsto t\e_{w_{i}}$, where we define $\e_{\overline{l}}:= -\e_{l}$ for $1 \leq l \leq n$. Also, in general, for paths $\pi_{1}, \cdots, \pi_{k}:[0,1] \rightarrow \mathfrak{h}_{\mathbb{R}}^{*}$, we have 

\begin{align*}
\op{res}(\pi_{1}*\cdots * \pi_{k}) = \op{res}(\pi_{1})* \cdots * \op{res}(\pi_{k}).
\end{align*}

\noindent
Set ${\tightoverset{\hbox{\textasciicircum}}{\e}}_{i}= \op{res}(\e_{i}) \hbox{ for } i \in \{1, \cdots, 2n\}.$ Then, for $i \in \{1, \cdots, 2n\}$ and $j \in \{1, \cdots, n\}$ we have

\[
    {\tightoverset{\hbox{\textasciicircum}}{\e}}_{i}({\tightoverset{\hbox{\textasciicircum}}{h}}_{j}) =  
\begin{cases}
    1& \text{ if } i \in \{j, 2n-j\}\\
    -1& \text{ if }  i \in \{j+1, 2n-j+1\}\\
    0& \text{ otherwise. }
\end{cases}
\]
\noindent
Therefore ${\tightoverset{\hbox{\textasciicircum}}{\e}}_{i} = - {\tightoverset{\hbox{\textasciicircum}}{\e}}_{2n - i +1}$, which means we can describe $\op{res}(\pi_{w})$ in the following simple way: First obtain from $w$ a word $\op{res}(w)$ in the alphabet $\mathcal{C}_{n}$ by replacing a letter $w_{i}$ in $w$ with $\overline{2n - w_{i} +1}$ if  $n< w_{i} \leq 2n$. All other letters stay the same. This all then implies that $\op{res}(\pi_{w}) = \pi_{\op{res}(w)}.$

\begin{ex}
\label{firstexampleofrestrictedpath}
Let $n = 2$ and $w = 121223341$. Then $\op{res}(w) = 12122\bar{2}\bar{2}\bar{1}1$. 

\end{ex}

\section{The Naito-Sagaki conjecture}

Let $\lambda \in \p^{+}_{A_{2n-1}}$ be dominant and let $\li(\lambda)$ be the associated simple module for $\mathfrak{sl}(2n,\mathbb{C})$. Recall the set  $\Gamma(\underline{d}_{\lambda})^{\op{SSYT}}$ of semi-standard Young tableaux of shape $\underline{d}_{\lambda}$. Let

\begin{align*}
\op{domres}(\lambda) = \{\delta \in \Gamma(\underline{d}_{\lambda})^{\op{SSYT}}: \op{res}(\pi_{w(\delta)}) \text{ is dominant}\}
\end{align*}

\noindent
be the subset of   $\Gamma(\underline{d}_{\lambda})^{\op{SSYT}}$, elements of which are those semi-standard Young tableaux whose words are dominant under restriction (considered as paths), and for $\nu \leq \lambda$, let 

\begin{align*}
\op{domres}(\lambda, \nu) = \{\delta \in \op{domres}(\lambda): \delta (1) = \nu\}.
\end{align*}

\begin{ex}
\label{domresexample}
Let $n = 2$ and $\lambda = \omega_{1} + \omega_{2}$. Then

\begin{align*}
\op{domres}(\lambda) &= \left\{\Skew(0:\mbox{\tiny{1}},\mbox{\tiny{1}}|0:\mbox{\tiny{2}}), \Skew(0:\mbox{\tiny{1}},\mbox{\tiny{1}}|0:\mbox{\tiny{4}})\right\}\\
\op{domres}(\lambda, \lambda) &= \left\{\Skew(0:\mbox{\tiny{1}},\mbox{\tiny{1}}|0:\mbox{\tiny{2}})\right\}\\
\op{domres}(\lambda, \omega_{1}) &= \left\{ \Skew(0:\mbox{\tiny{1}},\mbox{\tiny{1}}|0:\mbox{\tiny{4}})\right\}.
\end{align*}
\end{ex}

\begin{conj}\cite{branchconj}
\label{naitosagakiconjecture}
Let $\lambda \in \p_{A_{2n-1}}^{+}$ be dominant, and let $\li(\lambda)$ be the associated simple module for 
$\mathfrak{sl}(2n, \mathbb{C})$. Then 

\begin{align*}
\op{res}^{\mathfrak{g}}_{\mathfrak{g}^{\sigma}} (\li(\lambda)) = \underset{\delta \in \op{domres}(\lambda)}{\bigoplus} \li(\tightoverset{\hbox{\textasciicircum}}{\delta(1)}).
\end{align*}

\end{conj}

\begin{thm}[\cite{branchconj}]
\label{existingcases}
Conjecture \ref{naitosagakiconjecture} is true for $\lambda =  a\omega_{1} + \omega_{k}$ and $\lambda = a\omega_{k}$, $a \in \mathbb{Z}_{\geq 0}$.
\end{thm}

\begin{ex}
\label{firstexamplebranch}
Let $n = 2$ and $\lambda = \omega_{1} + \omega_{2}$ as in Example \ref{domresexample}. Then 

\begin{align*}
\op{res}^{\mathfrak{g}}_{\mathfrak{g}^{\sigma}}(\li(\lambda)) = \li({\tightoverset{\hbox{\textasciicircum}}{\omega}}_{1}) \oplus \li({\tightoverset{\hbox{\textasciicircum}}{\omega}}_{1} + {\tightoverset{\hbox{\textasciicircum}}{\omega}}_{2}).
\end{align*}
\end{ex}
\begin{rem}
Conjecture \ref{naitosagakiconjecture} is stated in \cite{branchconj} for $\li(\lambda)$ a representation of $\mathfrak{gl}(2n, \mathbb{C})$ for $\lambda$ non-negative and dominant. However, the representation of $\mathfrak{gl}(2n, \mathbb{C})$ induced by an irreducible representation of $\mathfrak{sl}(2n, \mathbb{C})$ has the same highest weight and restricts back to itself. See \S 15.3 in \cite{fultonharris}.
\end{rem}


\section{Littlewood-Richardson tableaux and n-symplectic Sundaram tableaux: branching}
\label{skewsection}
\begin{defn}
Let $\lambda, \nu \in \p_{A_{2n-1}}$ be two dominant weights such that $\underline{d}_{\nu} \subset \underline{d}_{\lambda}$ (this means that one shape is contained in the other when aligned with respect to their top left corners. In Example \ref{firstexamplensymplectic} below we see that $\underline{d}_{\nu} \subset \underline{d}_{\lambda}$). A tableau $\mathscr{T}$ of skew shape $\lambda / \nu$ is a filling of an arrangement of boxes of shape $\lambda$ leaving the boxes that belong to $\nu \subset \lambda$ blank, with the others having entries in the alphabet $\mathcal{A}_{2n}$, and such that these entries are strictly increasing along the columns. The word $w(\mathscr{T})$ of $\mathscr{T}$ is obtained just as for semi-standard Young tableaux, reading from right to left and from top to bottom, ignoring the blank boxes. 
\end{defn}

A shape $\underline{d} = (d_{1}, \cdots, d_{k})$ is \textbf{even} if $d_{i} = 0$ unless $i$ is an even number, for $i \in \{1, \cdots, k\}$.  Also, for a shape $\underline{d}$ define $l(\underline{d})$  to be the length of the longest column of the associated arrangement of boxes. 


\begin{defn}
Let $\lambda , \nu, \eta \in \p^{+}_{A_{2n-1}}$ be dominant weights such that the shapes $\underline{d}_{\nu}$ and $\underline{d}_{\eta}$ are contained in the shape $\underline{d}_{\lambda}$ of $\lambda$, and such that $\underline{d}_{\eta}$ is even. A \textbf{Littlewood-Richardson} (respectively \textbf{n-symplectic Sundaram} or just \textbf{Sundaram}) tableau of skew shape $\lambda / \nu$ and weight 
$\eta$ is a tableau of skew shape $\lambda / \nu$ that is semi-standard, and has a dominant word of weight $\eta$ (respectively $2i+1 $ does not appear strictly below row  $n+i$ for $i \in \{0, 1, \cdots, \frac{1}{2} l(\underline{d}_{\eta})\}$). Here a word $w \in \mathcal{W}(\mathcal{A}_{2n})$ is dominant if the path $\pi_{w}$ is dominant. We will denote them by $\op{LR}(\lambda/ \nu, \eta)$   (respectively $(\op{LRS}(\lambda/\nu, \eta)$.)
\end{defn}

\begin{rem}
\label{redundance1}
Note that if $l(\lambda) \leq n$ (such weights are called \textbf{stable}) then $\op{LRS}(\lambda/\nu, \eta) = \op{LR}(\lambda/\nu, \eta)$. 
\end{rem}

\begin{rem}
\label{redundance2}
If $\lambda$ is stable and $\mathscr{T}$ is a Littlewood-Richardson tableau of skew shape $\lambda / \nu$ then its entries belong to the set $\{1, \cdots, n\}$. This is because if, say, $k$ appears in row $l_{k}$ of $\mathscr{T}$, then, since the word of $\mathscr{T}$ is dominant, a $k-1$ must appear either directly above $k$ in the same column, or in a column to the right, and since $\mathscr{T}$ is semi-standard, it appears in at most row $l_{k} - 1$.
\end{rem}

\begin{ex}
\label{firstexamplensymplectic}
The tableau $\mathscr{L} = \Skew(0:\mbox{},\mbox{\tiny{1}},\mbox{\tiny{1}}|0:\mbox{ },\mbox{\tiny{2}}|0:\mbox{\tiny{2}})$ is  a Littlewood-Richardson tableau of skew shape $\lambda / \nu$ and weight $\eta$ for $\lambda = \omega_{1} + \omega_{2} + \omega_{3}, \nu = \omega_{2},$ and $\eta = 2 \omega_{2}$ and the tableau $\mathscr{T} = \Skew(0: \mbox{ }|0: \mbox{ } |0: \mbox{\tiny{1}})$ is a Littlewood-Richardson tableau of skew shape $\lambda'/ \nu'$ and weight $\eta'$ for $\lambda' = \omega_{3}, \nu' = \omega_{2},$ and $\eta' = \omega_{1}$. Notice that  $\mathscr{L}$ is 2-symplectic Sundaram while $\mathscr{T}$ is not. 
\end{ex}

\begin{defn}
The Littlewood-Richardson coefficient is defined as the number $c^{\lambda}_{\nu, \eta} \in \mathbb{Z}_{\geq 0}$ such that
\begin{align*}
\li(\nu) \otimes \li(\eta) = \underset{\nu \leq \lambda}{\bigoplus} c^{\lambda}_{\nu, \eta} \li(\lambda)
\end{align*}
\noindent
where $\li(\lambda), \li(\nu),$ and $\li(\eta)$ are all representations of $\mathfrak{sl}(2n, \mathbb{C})$. 
\end{defn}

\noindent
Theorem \ref{litrich} below is known as the \textbf{Littlewood-Richardson rule}. It was first stated in 1943 by Littlewood and Richardson. 

\begin{thm}\cite{howelee}
\label{litrich}
The Littlewood-Richardson coefficients are obtained by counting Littlewood-Richardson tableaux: 
\begin{align*}
c^{\lambda}_{\nu, \eta} = |\op{LR}(\lambda/ \nu, \eta)|.
\end{align*}
\end{thm}

\begin{rem}
\label{redundance}
Theorem \ref{litrich} implies that $c^{\lambda}_{\nu, \eta} = c^{\lambda}_{\eta, \nu}$. 
\end{rem}

We will use the notation $c^{\lambda}_{\nu, \eta}(\s) = |\op{LRS(\lambda/ \nu, \eta)}|$. The following theorem was proven by Sundaram in Chapter IV of her PhD thesis \cite{sundaram}. See also Corollary 3.2 of \cite{sundarampaper}. For stable weights it was proven by Littlewood in \cite{litt} and is known as the Littlewood branching rule. 


\begin{thm} \cite{sundaram}
\label{sundaramstheorem}
Let $\lambda \in \p_{A_{2n-1}}$ be dominant. Then 

\begin{align*}
\op{res}^{\mathfrak{g}}_{\mathfrak{g}^{\sigma}}(\li(\lambda)) = \underset{l(\underline{d}_{\nu}) \leq n}{\underset{\underline{d}_{\nu} \subset \underline{d}_{\lambda}}{\bigoplus}} \N_{\lambda, \nu} \li(\nu)
\end{align*}

\noindent
where 
\begin{align*}
\N_{\lambda, \nu} = \underset{\underline{d}_{\eta} \hbox{ \tiny{even} }}{\sum}c^{\lambda}_{\nu, \eta}(\s). 
\end{align*}
\end{thm}

\section{Proof of the Naito-Sagaki conjecture for $n = 2$ and $\lambda = a_{1}\omega_{1}+a_{2}\omega_{2}+a_{3}\omega_{3}$.}

As its title suggests, in this section we give a proof of Conjecture \ref{naitosagakiconjecture} in the cases  $n = 2$ and $\lambda = a_{1}\omega_{1}a_{2}+\omega_{2}+a_{3}\omega_{3}$, for all $n$. We will do so using Theorem \ref{sundaramstheorem} from  Section \ref{skewsection}. The following construction should provide some insight. Given a tableau $\mathscr{T} \in \op{domres}(\lambda)$ we will construct a weight $\eta_{\mathscr{T}}$ with even shape $\underline{d}_{\eta_{\mathscr{T}}}$. To do this, first replace, in $\mathscr{T}$, all letters $w$ such that $n< w \leq 2n$ by $\overline{2n-w+1}$, and denote the resulting symplectic semi-standard tableau by $\op{res}(\mathscr{T})$. Its word $w(\op{res}(\mathscr{T}))$ is equal to the restricted word $\op{res}(w(\mathscr{T}))$. Now, in each column of $\op{res}(\mathscr{T})$, replace an entry $w$ by a blank square if $\overline{w}$ appears to its left in the word  $\op{res}(w(\mathscr{T})) = w(\op{res}(\mathscr{T}))$. In that case, replace the entry $\overline{w}$ by a blank square as well. Count the number of blank squares in each column, and order these squares to obtain an arrangement of boxes of  shape $\underline{d}_{\eta_{\mathscr{T}}}$.

\begin{ex}
\label{exevenshape}
Let $n = 2$ and 
\begin{align*}
\mathscr{T} = \Skew(0:\mbox{\tiny{1}},\mbox{\tiny{1}},\mbox{\tiny{1}}|0:\mbox{\tiny{2}},\mbox{\tiny{2}}|0:\mbox{\tiny{3}}).
\end{align*} 

Then 
\begin{align*}
\op{res}(\mathscr{T}) = \Skew(0:\mbox{\tiny{1}},\mbox{\tiny{1}},\mbox{\tiny{1}}|0:\mbox{\tiny{2}},\mbox{\tiny{2}}|0:\mbox{\tiny{$\overline{2}$}}),\mbox{      } \eta_{\mathscr{T}} = \omega_{2}, \hbox{ and } \underline{d}_{\eta_{\mathscr{T}}} = (0,1)
\end{align*}

\noindent
where the arrangement of boxes $\Skew(0: \mbox{ }|0: \mbox{ })$ of shape $\underline{d}_{\eta_{\mathscr{T}}}$ is obtained by replacing $\Skew(0: \mbox{\tiny{2}}|0: \mbox{\tiny{$\overline{2}$}})$ in $\op{res}(\mathscr{T})$ by blank squares.
\end{ex}

\begin{lem}
\label{evennumberofboxes}
Let $\mathscr{T}$ be as above. Then the shape $\underline{d}_{\eta_{\mathscr{T}}}$ is even. 
\end{lem}

\begin{rem}
Lemma \ref{evennumberofboxes} is only true for $\mathscr{T}$ a semi-standard Young tableau. Consider for example $n=2$ and the key (as in \cite{knuth}, \cite{jt}) $\mathscr{T} = \Skew(0:\hbox{\tiny{1}}, \hbox{\tiny{4}}, \hbox{\tiny{1}})$. Then $\op{res}(\mathscr{T}) = \Skew(0:\hbox{\tiny{1}}, \hbox{\tiny{$\bar 1$}}, \hbox{\tiny{1}})$ is dominant, however, the shape $\eta_{\mathscr{T}} = (1,0)$, which is not even.
\end{rem}

\begin{proof}
We will call a column standard if its entries are consecutive integers, starting with $1$. The proof is by induction on the number of right-most aligned consecutive standard columns in $\mathscr{T}$, counted from right to left. For example: the tableau $\Skew(0:\mbox{\tiny{1}},\mbox{\tiny{1}},\mbox{\tiny{1}}|0:\mbox{\tiny{2}},\mbox{\tiny{2}}|0:\mbox{\tiny{4}})$ has two right-most aligned consecutive standard columns, the tableau $\Skew(0:\mbox{\tiny{1}},\mbox{\tiny{1}},\mbox{\tiny{1}}|0:\mbox{\tiny{2}},\mbox{\tiny{4}}|0:\mbox{\tiny{3}})$ has one, and the tableau $\Skew(0:\mbox{\tiny{1}},\mbox{\tiny{1}}|0:\mbox{\tiny{2}},\mbox{\tiny{4}}|0:\mbox{\tiny{3}})$ has none. Let $\mathscr{D}$ be the first column in $\op{res}(\mathscr{T})$ (counted from right to left) that is not standard. Then there exists $s>0$ such that the first $s$ boxes of $\mathscr{D}$ are filled in with the numbers $i$ such that $i \leq s$, and its $s+1$-th box is filled in with $\overline{l}$ for some $l\leq n$. Since $\op{res}(\mathscr{T})$ has a dominant word, it must even hold that $l\leq s$. The same is true for the rest of the entries in $\mathscr{D}$, which are barred since entries are strictly increasing. The boxes in $\mathscr{D}$ with barred entries together with the boxes in $\mathscr{D}$ that have as entries their non-barred versions (they all exist, since the word of $\mathscr{T}$ is dominant) make up one of the columns of the arrangement of boxes of shape $\underline{d}_{\eta}$. This column has an even number of boxes.  Let us now ignore these entries. Let $\mathscr{C}$ be the closest column to the left of $\mathscr{D}$ that is not standard. For the induction step, we construct a new semi-standard Young tableau in which $\mathscr{C}$ is the first column to not be standard. Let $\mathscr{L}$ be the semi-standard Young tableau made up of the first right-most aligned standard columns of $\mathscr{T}$. Since the word of $\op{res}(\mathscr{T})$ is dominant, we may use the non-ignored entries in $\mathscr{D}$ to construct a new tableau $\mathscr{L}'$ from $\mathscr{L}$. We do this by adding one of these boxes either at the end of $\mathscr{L}$ or at the end of a column, in such a way that the resulting arrangement is still a semi-standard Young tableau with standard columns only. For example,  if $n= 3$ and $\mathscr{T} = \Skew(0:\mbox{\tiny{1}},\mbox{\tiny{1}},\mbox{\tiny{1}},\mbox{\tiny{1}},\mbox{\tiny{1}}|0:\mbox{\tiny{2}},\mbox{\tiny{2}},\mbox{\tiny{2}},\mbox{\tiny{2}}|0:\mbox{\tiny{3}},\mbox{\tiny{3}},\mbox{\tiny{3}}|0:\mbox{\tiny{4}},\mbox{\tiny{5}}|0:\mbox{\tiny{6}})$, then $\op{res}(\mathscr{T}) = \Skew(0:\mbox{\tiny{1}},\mbox{\tiny{1}},\mbox{\tiny{1}},\mbox{\tiny{1}},\mbox{\tiny{1}}|0:\mbox{\tiny{2}},\mbox{\tiny{2}},\mbox{\tiny{2}},\mbox{\tiny{2}}|0:\mbox{\tiny{3}},\mbox{\tiny{3}},\mbox{\tiny{3}}|0:\mbox{\tiny{$\overline{3}$}},\mbox{\tiny{$\overline{2}$}}|0:\mbox{\tiny{$\overline{1}$}})$ has a dominant word. Also, $\mathscr{L} = \Skew(0:\mbox{\tiny{1}},\mbox{\tiny{1}},\mbox{\tiny{1}}|0:\mbox{\tiny{2}},\mbox{\tiny{2}}|0:\mbox{\tiny{3}}), \mathscr{D} = \Skew(0:\mbox{\tiny{1}}|0:\mbox{\tiny{2}}|0:\mbox{\tiny{3}}|0:\mbox{\tiny{$\overline{2}$}}), \mathscr{C} = \Skew(0:\mbox{\tiny{1}}|0:\mbox{\tiny{2}}|0:\mbox{\tiny{3}}|0:\mbox{\tiny{4}}|0:\mbox{\tiny{6}}), \mathscr{L}' = \Skew(0:\mbox{\tiny{1}},\mbox{\tiny{1}},\mbox{\tiny{1}},\mbox{\tiny{1}}|0:\mbox{\tiny{2}},\mbox{\tiny{2}}|0:\mbox{\tiny{3}},\mbox{\tiny{3}})$, and the ignored boxes in $\mathscr{D}$ are $\Skew(0:\mbox{\tiny{2}})$ and $\Skew(0:\mbox{\tiny{$\overline{2}$}})$. Define a new tableau $\op{res}(\mathscr{T})'$ by concatenating all the columns in $\op{res}(\mathscr{T})$, from left to right and up to $\mathscr{D}$, with $\mathscr{L}'$. It follows from the construction that $\op{res}(\mathscr{T})'$ has a dominant word and is semi-standard. In the previous example, this new tableau is $\op{res}(\mathscr{T})' = \Skew(0:\mbox{\tiny{1}},\mbox{\tiny{1}},\mbox{\tiny{1}},\mbox{\tiny{1}},\mbox{\tiny{1}}|0:\mbox{\tiny{2}},\mbox{\tiny{2}},\mbox{\tiny{2}}|0:\mbox{\tiny{3}},\mbox{\tiny{3}},\mbox{\tiny{3}}|0:\mbox{\tiny{$\overline{3}$}}|0:\mbox{\tiny{$\overline{1}$}})$. We may then apply the induction hypothesis. This concludes the proof. 
\end{proof}

\begin{lem}
\label{easylemma}
If $\lambda = a_{1} \omega_{1} + a_{2}\omega_{2} + a_{3} \omega_{3}$
and $\nu$ and $\eta$ are dominant weights in $\p_{A_{2n-1}}$ such that $\underline{d}_{\eta}$ is even and  both $\underline{d}_{\eta}$ and $\underline{d}_{\nu}$ are contained in $\underline{d}_{\lambda}$, then 
\begin{align*}
c^{\lambda}_{\nu, \eta}(\s) = c^{\lambda}_{\nu, \eta}.
\end{align*}
\end{lem}

\begin{proof}
Assume that $\mathscr{T}$ is a Littlewood-Richardson tableau of skew shape $\lambda / \nu$ and weight $\eta$ that is not Sundaram. This means that there is at least a ``1'' in the third row. Since $\mathscr{T}$ is semi-standard, all the ``1'''s in the third row must appear left-most and all next to one another. But since $\underline{d}_{\eta}$ is even, for each of these 1's there must exist a ``2''  that appears before it , in the word reading order. But this means, since the word is dominant, that there must have appeared a 1 before this ``2''. This contradicts the evenness of $\underline{d}_{\eta}$!
\end{proof}

\begin{thm}
\label{conny2}
The Naito-Sagaki conjecture is true for $n = 2$ and for any $n$ if $\lambda = a_{1}\omega_{1} + a_{2}\omega_{2} + a_{3}\omega_{3}$.
\end{thm}

\begin{proof}
Fix $\lambda$ and $\nu$ as in Lemma \ref{easylemma} above. Then, for all $\eta$ such that $\underline{d}_{\eta}$ is even, $c^{\lambda}_{\nu, \eta}(\s) = c^{\lambda}_{\eta, \nu}$, by Lemma \ref{easylemma} and Remark \ref{redundance}. The following proposition therefore proves the theorem.
\end{proof}

\begin{prop}
\label{theclaim!}
In case the assumptions of Theorem \ref{conny2} hold, there is a bijection
\begin{align*}
\op{domres}(\lambda, \nu) \overset{1:1}{\longleftrightarrow} \underset{\underline{d}_{\eta} \hbox{ \tiny{even} }}{\underset{\underline{d}_{\eta} \subseteq \underline{d}_{\lambda};}{\bigcup}} \op{LR}(\lambda/\eta, \nu)
\end{align*}
\end{prop}

\begin{proof}[Proof of Proposition \ref{theclaim!}]
Let  $ \mathscr{T} \in \op{domres}(\lambda, \nu)$, and set 
\begin{align*}
m_{1} &= \hbox{ \# columns in} \mathscr{T} \hbox{ of the form } \Skew(0: \mbox{\tiny{1}}|0: \mbox{\tiny{2}} |0: \mbox{\tiny{2n}}),\\
m_{2} &= \hbox{ \# columns in }\mathscr{T}\hbox{ of the form } \Skew(0: \mbox{\tiny{1}}|0: \mbox{\tiny{2}} |0: \mbox{\tiny{\mbox{\tiny{2n-1}}}}), \\
m_{3} &= \hbox{  \# columns in }\mathscr{T}\hbox{ of the form } \Skew(0: \mbox{\tiny{1}}|0: \mbox{\tiny{2n}})
\hbox{ and }\\
m_{4} &= \hbox{  \# columns in }\mathscr{T}\hbox{ of the form \Skew(0: \mbox{\tiny{1}}|0: \mbox{\tiny{2}} |0: \mbox{\tiny{3}})} .
\end{align*}
\noindent
If $n = 2$ then $m_{2} = m_{4}$. It follows from semi-standardness and dominance of $\op{res}(w(\mathscr{T}))$ that these are the only possible columns aside from columns of the form $\Skew(0: \mbox{\tiny{1}}|0: \mbox{\tiny{2}})$ and the single box columns $\Skew(0:1)$. Note that since $\op{res}(w(\mathscr{T}))$ is dominant, the following condition  holds
\begin{align}
\label{domi}
m_{1} \leq \lambda_{1} - \lambda_{2}. 
\end{align}
\noindent
Actually (\ref{domi}) is equivalent to the dominance of $\op{res}(w(\mathscr{T}))$, once the $b_{i}$ are set. We assign to $\mathscr{T}$ a Littlewood-Richardson tableau $\varphi(\mathscr{T}) \in \op{LR}(\lambda/\eta_{\mathscr{T}}, \nu)$. By Lemma \ref{evennumberofboxes}, $\eta_{\mathscr{T}}$ is even. Write 

\begin{align*}
\lambda = \lambda_{1} \e_{1} + \lambda_{2}\e_{2} + \lambda_{3}\e_{3}.
\end{align*}

\noindent
 Note that $\eta_{\mathscr{T}}$ has $m = m_{1}+m_{2} +m_{3}$ columns, all of length 2. Fill in the first $\lambda_{1} - m$ right-most boxes in the first row with a ``1'' , and the first $\lambda_{2} - b$ right-most boxes in the second row with a ``2''. If $n \neq 2$ fill in the first $m_{4}$ right-most entries of the third row with a ``3''. Then \label{conn2}fill in the next right-most $m_{1}$ entries in the third row with a ``2'', and the remaining entries with a ``1''.  The resulting tableau $\varphi(\mathscr{T})$ is a Littlewood-Richardson tableau by construction. Now we will show that any element in $ \underset{\eta \hbox{ \tiny{even} }}{\underset{\underline{d}_{\eta} \subset \underline{d}_{\lambda};}{\bigcup}} \op{LR}(\lambda/\eta, \nu)$ can be obtained in this way.
Let $\eta \in \p^{+}_{A_{2n-1}}$ have an even shape $\underline{d}_{\eta} \subset \underline{d}_{\lambda}$ (this means $\underline{d}_{\eta}$ consists of size 2 columns)  and let $\mathscr{L} \in \op{LR}(\lambda/\eta, \nu)$. Set

\begin{align*}
l_{1} &= \hbox{ \# of 1's in } \mathscr{L} \\
l_{2} &= \hbox{ \# of 2's in } \mathscr{L}, \hbox{ and } \\
b &=  \hbox{\# of columns of } \eta.
\end{align*} 

\noindent
Note that this information determines $\mathscr{L}$ together with $\lambda$. In view of the previous construction, we would like to find non-negative integers $m_{1}, m_{2}, m_{3}$ and $m_{4}$ such that 

\begin{align}
\label{1inverseconditionn2}
l_{1} &= \lambda_{1} - m_{1} - m_{3} \\
\label{2inverseconditionn2}
l_{2} &= \lambda_{2} - m_{3} -m_{2} \\
\label{3inverseconditionn2}
m &= m_{1} + m_{2} + m_{3}\\
m_{4} & = \hbox{ \# of 3's in } \mathscr{L}
\end{align}

\noindent
Since $\mathscr{L}$ has a dominant word, we have

\begin{align}
\label{dominverse}
l_{2} - (\lambda_{2} - b) \leq \lambda_{1} - \lambda_{2}.
\end{align}


\noindent
Substituting (\ref{2inverseconditionn2}) and (\ref{3inverseconditionn2}) in (\ref{dominverse}) we get precisely (\ref{domi}), so if we find solutions $m_{1},m_{2}, m_{3}$, the resulting tableau will automatically belong to $\op{domres}(\lambda, \nu)$. Claim \ref{systemofeqnssolis} below assures that this is the case.
\end{proof}

\begin{cl}
\label{systemofeqnssolis}
The system determined by (\ref{1inverseconditionn2}), (\ref{2inverseconditionn2}), and (\ref{3inverseconditionn2}), has integer solutions (possibly zero) $b_{1}, b_{2}$ and $b_{3}$ if and only if 
\end{cl}

\begin{align}
\label{finalexistence1}
m &\geq \lambda_{1} - l_{1} \\
\label{finalexistence2}
m &\geq \lambda_{2} -  l_{2}\\
\label{finalexistence3}
\lambda_{1} + \lambda_{2} &\geq m +l_{1} + l_{2}.
\end{align}

\noindent
It follows from the definitions that these conditions are satisfied by all elemen\label{conn2}ts of

\begin{align*}
\underset{\underline{d}_{\eta} \hbox{ \tiny{even} }}{\underset{\underline{d}_{\eta} \subset \underline{d}_{\lambda};}{\bigcup}} \op{LR}(\lambda/\eta, \nu).
\end{align*}

\noindent
To conclude we give a proof of Claim \ref{systemofeqnssolis}. 
\begin{proof}[Proof of Claim \ref{systemofeqnssolis}]
We need to solve the system of equations determined by (\ref{1inverseconditionn2}), (\ref{2inverseconditionn2}), and (\ref{3inverseconditionn2}). From (\ref{1inverseconditionn2}) and (\ref{3inverseconditionn2}) we have 

\begin{align}
\label{1eqnsn2b2}
m_{1} &=\lambda_{1} - l_{1} - m_{3}\\
\label{eqnsn2b2}
m_{2} &= m - \lambda_{1} + l_{1} - \lambda_{1}
\end{align}
\noindent
 Therefore $ m_{2} \geq 0$ if and only (\ref{finalexistence1}) holds. Substituting (\ref{eqnsn2b2}) into (\ref{2inverseconditionn2}) we get 
 
 \begin{align}
 \label{b3equation}
 m_{3} = \lambda_{1} + \lambda_{2} - l_{1} - l_{2} - m
 \end{align}
 \noindent
  Hence $m_{3} \geq 0$ if and only if (\ref{finalexistence3}) holds. Now substitute (\ref{b3equation}) into (\ref{1eqnsn2b2}) and get 
  
\begin{align}
m_{1} = l_{2} - \lambda_{2} + m,
\end{align}
\noindent
and hence $m_{1} \geq 0$ if and only if (\ref{finalexistence2}) holds. This concludes the proof of Claim \ref{systemofeqnssolis}.
\end{proof}

\section{Questions and perspectives}
The bijection from Proposition \ref{theclaim!} (for any $n$) would imply Conjecture \ref{naitosagakiconjecture} for a stable weight $\lambda$. A more general bijection

\begin{align*}
\op{domres}(\lambda, \nu) \overset{1:1}{\longleftrightarrow} \underset{\underline{d}_{\eta} \hbox{ \tiny{even} }}{\underset{\underline{d}_{\eta} \subseteq \underline{d}_{\lambda};}{\bigcup}} \op{LRS}(\lambda/\eta, \nu)
\end{align*}

\noindent
would imply Conjecture \ref{naitosagakiconjecture} for any $n$ and any $\lambda$. The questions of how to establish such a bijection in general, and of how to generalize Conjecture \ref{naitosagakiconjecture} to other types, are the motivation for \cite{beajaz}.

\end{document}